\documentclass{amsart}
\usepackage{amssymb,amsmath,amscd,amsthm}
\usepackage{amsfonts,epsfig,latexsym,graphicx,amssymb}

\date{\today}

\newcommand{\Z}{{\mathbb Z}}
\newcommand{\R}{{\mathbb R}}
\newcommand{\C}{{\mathbb C}}

\newcommand{\T}{{\mathbb T}}
\newcommand{\Q}{{\mathbb Q}}

\newtheorem{theorem}{Theorem}[section]
\newtheorem{lemma}[theorem]{Lemma}
\newtheorem{prop}[theorem]{Proposition}

\theoremstyle{definition}
\newtheorem{remark}[theorem]{Remark}
\newtheorem{question}[theorem]{Question}

\newcommand{\E}{{\mathbb E}\,}

\def\be{\begin{equation}}
\def\ee{\end{equation}}

\renewcommand{\includegraphics}[2][]{}

\begin{document}

%\title[LP Schr\"odinger Operators with PP Spectrum and Zero LE Are Dense]{A Dense Set of Limit-Periodic Schr\"odinger Operators Have Pure Point Spectrum and Zero Lyapunov Exponents}

\title[An Extension of the Kunz-Soulliard Approach to Localization]{An Extension of the Kunz-Souillard Approach to Localization in One Dimension and Applications to Almost-Periodic Schr\"odinger Operators}

\author[D.\ Damanik]{David Damanik}

\address{Department of Mathematics, Rice University, Houston, TX~77005, USA}

\email{damanik@rice.edu}

\thanks{D.\ D.\ was supported in part by NSF grant DMS--1361625.}

\author[A.\ Gorodetski]{Anton Gorodetski}

\address{Department of Mathematics, University of California, Irvine, CA~92697, USA}

\email{asgor@math.uci.edu}

\thanks{A.\ G.\ was supported in part by NSF grant DMS--1301515. }% and %DMS--0901627 and IIS-1018433.}

\begin{abstract}
We generalize the approach to localization in one dimension introduced by Kunz-Souillard, and refined by Delyon-Kunz-Souillard and Simon, in the early 1980's in such a way that certain correlations are allowed. Several applications of this generalized Kunz-Souillard method to almost periodic Schr\"odinger operators are presented.

On the one hand, we show that the Schr\"odinger operators on $\ell^2(\Z)$ with limit-periodic potential that have pure point spectrum form a dense subset in the space of all limit-periodic Schr\"odinger operators on $\ell^2(\Z)$. More generally, for any bounded potential, one can find an arbitrarily small limit-periodic perturbation so that the resulting operator has pure point spectrum. Our result complements the known denseness of absolutely continuous spectrum and the known genericity of singular continuous spectrum in the space of all limit-periodic Schr\"odinger operators on $\ell^2(\Z)$.

On the other hand, we show that Schr\"odinger operators on $\ell^2(\Z)$ with arbitrarily small one-frequency quasi-periodic potential may have pure point spectrum for some phases. This was previously known only for one-frequency quasi-periodic potentials with $\|\cdot\|_\infty$ norm exceeding $2$, namely the super-critical almost Mathieu operator with a typical frequency and phase. Moreover, this phenomenon can occur for any frequency, whereas no previous quasi-periodic potential with Liouville frequency was known that may admit eigenvalues for any phase.
\end{abstract}

\maketitle

\section{Introduction}\label{s.intro}

We consider Schr\"odinger operators $H$ acting on $\ell^2(\Z)$ via
\begin{equation}\label{e.oper}
[H \psi](n) = \psi(n+1) + \psi(n-1) + V(n) \psi(n).
\end{equation}
We will assume throughout that the potential $V : \Z \to \R$ is bounded. In this case the operator $H$ is bounded and self-adjoint and hence its spectrum $\sigma(H)$ is a compact subset of $\R$ and, for every $\psi \in \ell^2(\Z)$, there is an associated spectral measure $d\rho = d\rho_{H,\psi}$ such that $\langle \psi, g(H) \psi \rangle = \int g \, d\rho$ for bounded measurable functions $g : \R \to \C$. One says that $H$ has purely absolutely continuous spectrum (resp., purely singular continuous spectrum, or pure point spectrum) if all spectral measures are absolutely continuous (resp., singular continuous, or pure point). The basic spectral analysis of the operator $H$ consists of identifying or describing the spectrum and determining the type of the spectral measures.

Ergodic potentials are of significant interest from the point of view of solid state physics. The extreme examples are given by periodic and random potentials. In between these extremes, almost periodic potentials have also been studied extensively. A potential $V$ is called periodic if there is $p \in \Z_+$ (the period) such that for all $n \in \Z$, we have $V(n+p) = V(n)$. A random potential is given by a sequence of independent identically distributed random variables. Formally, one may choose a non-degenerate probability measure $\nu$ on $\R$ with compact support $S \subset \R$, set $\Omega = S^\Z$, and form the product measure $\mu = \nu^\Z$. Then for every $\omega \in \Omega$, we consider the potential $V_\omega : \Z \to \R$ given by $V_\omega(n) = \omega_n$ and the associated operator $H_\omega$. Finally, almost periodic potentials $V$ are those whose translates $V_m$, $m \in \Z$, defined by $V_m(n) = V(n-m)$ form a relatively compact subset of $\ell^\infty(\Z)$. Here one also naturally forms $\Omega = \overline{\{V_m : m \in \Z\}}^{\ell^\infty(\Z)}$ and considers an associated family of operators $\{H_\omega\}_{\omega \in \Omega}$. There are two important subclasses of almost periodic potentials, namely those that are limit-periodic (i.e., those that are uniform limits of periodic potentials) and those that are quasi-periodic (i.e., when $\Omega$ is isomorphic to a closed subgroup of a finite-dimensional torus, i.e., to the direct sum of a finite-dimensional torus with a finite abelian group).

The spectral theory of these classes of Schr\"odinger operators has been reviewed recently in \cite{D15}. Let us mention some highlights. In the case of a periodic potential of period $p$, the spectrum is always a finite union of at most $p$ compact intervals. In the random case, the spectrum is the same set for $\mu$-almost every $\omega \in \Omega$, and this almost sure spectrum is again a finite union of compact intervals. In the almost periodic case, the spectrum is typically a Cantor set, in particular, it is nowhere dense. The spectral measures are purely absolutely continuous in the periodic case, they are for $\mu$-almost every $\omega \in \Omega$ pure point in the random case, and every one of the three basic spectral types can arise for almost periodic potentials. In fact, one can observe all three basic spectral types already within each of the classes of limit-periodic potentials as well as quasi-periodic potentials.

\subsection{Limit-Periodic Potentials}
In this paper we are especially interested in Schr\"odinger operators with limit-periodic potentials; see, for example, \cite{A09, DG10, DG11, DGa11, DLY15, F15, KG11} for recent progress in this area. It was already known that for a dense set of limit-periodic potentials, the spectrum is purely absolutely continuous, while for a generic (i.e., a dense $G_\delta$) set of limit-periodic potentials, the spectrum is purely singular continuous. Up until now it had been an intriguing open question whether pure point spectrum is also a dense phenomenon in the class of limit-periodic potentials. Note that if this is true, this would be the most surprising of the three statements, and likely the hardest to prove. This is because by their very definition, limit-periodic potentials are viewed as limits of a sequence of periodic potentials, and, as was pointed out above, periodic potentials always lead to purely absolutely continuous spectrum. Thus the natural inclination is to push this spectral type through to the limit, and this is precisely what the early papers on these operators sought to accomplish. It was realized only recently that the generic behavior, however, is purely singular continuous spectrum. Point spectrum, on the other hand, is farthest from the ``natural'' absolutely continuous spectral type, and hence showing that it occurs on a dense set should be regarded as surprising, and indeed proving it must use something other than mere periodic approximation.

In this paper we will answer this question which arose naturally from the known results, and which was asked in print on several occasions (e.g., \cite{D15, DGa11}).

\begin{theorem}\label{t.main1}
For every limit-periodic $V$ and every $\varepsilon > 0$, there is a limit-periodic $\tilde V$ with $\|V - \tilde V\|_\infty < \varepsilon$ such that the Schr\"odinger operator with potential $\tilde V$ has pure point spectrum.
\end{theorem}

We see that not only are all three basic spectrum types possible in the limit-periodic context, but each of them occurs on a dense set!

We would like to mention the paper \cite{MC84} by Molchanov and Chulaevsky which considered continuum limit-periodic Schr\"odinger operators and states a theorem that includes a statement like the one in Theorem~\ref{t.main1}. Unfortunately, it does not contain a proof of this result, and, to the best of our knowledge, a proof has never appeared in print. Since the result is important in light of the discussion above, we felt compelled to work out a detailed proof of it. Given the discussion in \cite{MC84} of their theorem it is quite clear that the strategy we use to prove Theorem~\ref{t.main1} is different from what is envisioned in \cite{MC84}.\footnote{We use the Kunz-Souillard method while Molchanov and Chulaevsky seem to rely on F\"urstenberg's theorem. It is not clear to us whether the latter strategy can easily be turned into a proof, as there appear to be several serious technical obstacles (such as quantitative control of transfer matrix growth, and the matching of the two decaying half-line solutions corresponding to the Oseledets directions). The Kunz-Souillard theory for the continuum setting (which was considered in \cite{MC84}) was fully developed only in 2011 \cite{DS11}.}

How does one go about proving such a result? As indicated above, periodic potentials lead to absolutely continuous spectrum, random potentials lead to pure point spectrum, and limit-periodic potentials are obtained as limits of periodic potentials. The fundamental idea is to use randomness in the approximation. Of course, one still needs periodic approximants, and hence one can only use randomness on finite intervals, which will then have to be repeated periodically. Thus, on the finite regions where one has randomness, one wants to build up signatures of localization that survive in the limit. Since one has a sequence of periodic approximants, one can take larger and larger intervals on which one has randomness, and hence exhaust space in this way. Of course, the strength of the added randomness has to decrease through this sequence sufficiently fast to ensure convergence of the periodic approximants. It is thus a combination of local randomness and global periodicity that ensures success, the former allowing for pure point spectrum and the latter for limit-periodicity.

Our proof is such that we can perturb arbitrary bounded potentials $V$, not only limit-periodic ones. Thus, the following strengthening of Theorem~\ref{t.main1} holds.

\begin{theorem}\label{t.main2}
For every bounded $V$ and every $\varepsilon > 0$, there is a limit-periodic potential $V_\mathrm{lp}$ with $\|V_\mathrm{lp}\|_\infty < \varepsilon$ such that the Schr\"odinger operator with potential $V + V_\mathrm{lp}$ has pure point spectrum.
\end{theorem}

As a result this is interesting, as up to now only random potentials were known to dominate any given background potential in this way. That is, it was known that an arbitrarily small random perturbation can turn a bounded Schr\"odinger operator into one with pure point spectrum. Theorem~\ref{t.main2} shows that this already holds with limit-periodic perturbations. Admittedly, the reason behind this phenomenon still has to do with randomness, but this happens on a local scale, while on a global scale we have turned the perturbations into ones that are far more ordered.

To implement and study local randomness we will employ a scheme that was originally developed by Kunz and Souillard in \cite{KS80}. It was refined and extended by Delyon, Kunz and Souillard \cite{DKS83} and Simon \cite{S82}. It is a rather unique approach, as it aims at dynamical localization, from which spectral localization (pure point spectrum with suitable eigenfunction decay estimates) is then deduced. Most other approaches to localization (especially in the context of almost periodic potentials, which is of primary interest in this paper) first establish spectral localization and then, with additional work, dynamical localization. Another feature of this approach, one that is particularly important for our purpose, is that it does not make any use of ergodicity. This allows one to add an arbitrary background potential. Moreover, the method is such that it processes randomness by going from site to site. It is thus by its very nature a local method that produces estimates step by step as it proceeds through an interval of randomness. Finally, it is an important feature but somewhat irrelevant for our proof that the identical distribution of the random variables is absolutely unnecessary for this method. This again is in stark contrast to all other methods of studying random potentials. Here we will content ourselves with rescaled versions of the same single-site distribution and hence we will not make full use of this particular feature of the method.

At first glance this suggests the following approach. Start with a background potential. Add a small random potential, but take only a finite portion of it, which is then repeated periodically. Process this finite interval of randomness via the Kunz-Souillard method and produce estimates. Add an even smaller random potential, but cut out an even longer piece, which is then repeated periodically. Process this finite interval of randomness via the Kunz-Souillard method and produce estimates. Continue in this way, making sure that the sum of the $\| \cdot \|_\infty$ norms of the perturbations add up to less than $\varepsilon$.

A straightforward attempt to implement the approach described in the previous paragraph fails, as there is too little control over the estimates one obtains from step to step so that one is left with nothing after passing to the limit. Instead we have found a rather elegant way to circumvent the infinite iteration of this process and work with a single sweep of the potential via Kunz-Souillard. Rather than adding random perturbations step by step, we add them all at once, and then process the resulting object on a partition of $\Z$ that is adapted to the parameters of the perturbations.

\subsection{Quasi-Periodic Potentials}

The theme of this paper is to use the Kunz-Souillard method to produce Schr\"odinger operators with pure point spectrum for which it had been hitherto unclear whether pure point spectrum is possible for such operators. We want to discuss one more class of potentials where we prove a result that follows this theme, namely small one-frequency quasi-periodic potentials. In the ergodic setting, choose $\Omega = \T = \R / \Z$, $\alpha \in \T \setminus \Q$, $T \omega = \omega + \alpha$, $f \in C(\T,\R)$. This gives rise to one-frequency quasi-periodic potentials $V_\omega(n) = f(T^n \omega) = f(\omega + n \alpha)$. This class of potentials has been extensively studied; compare again the survey \cite{D15}, and also \cite{JM15}. The case of analytic $f$ is especially well understood \cite{A14c, A14e, A14g, BG00}. Moreover the generic behavior in the class of continuous $f$ is understood as well \cite{AD05, BD08}. For regularity between analytic and continuous, there are several results (see, e.g., \cite{Bj05, C12, Ch12, E97, JN11, K05, S87, WZ15, WZ16}), but our understanding is much more limited. Generally speaking, small $f$ may lead to absolutely continuous spectrum, large $f$ may lead to pure point spectrum, but these statements require Diophantine properties of $\alpha$ and relatively high regularity of $f$ as they are known to fail otherwise. Concretely, no operator in this class with $\|f\|_\infty \le 2$ (and any $\alpha$) is known that has some eigenvalues,\footnote{On the other hand, if $V(n) = (2 + \varepsilon) \cos (2\pi(n\alpha + \omega))$, then for every $\varepsilon > 0$ and Lebesgue almost every $\alpha$ and $\omega$, the associated Schr\"odinger operator has pure point spectrum \cite{J99}.} and also no operator in this class with $\alpha$ Liouville (and any $f$) is known that has some eigenvalues. Here we prove the following theorem.%\marginpar{What are the spectral properties for other $\omega$'s in this family?}

\begin{theorem}\label{t.main4}
Given $\alpha \in \T \setminus \Q$, $\varepsilon > 0$ and $\omega \in \T$, there is $f \in C(\T,\R)$ with $\|f\|_\infty < \varepsilon$ such that the Schr\"odinger operator with potential $V(n) = f(\omega + n\alpha)$ has pure point spectrum.
\end{theorem}

In particular, even for Liouville frequencies we can find arbitrarily small sampling functions for which eigenvalues can occur. Admittedly, we prove this only for one phase $\omega$ (but by shift invariance of the conclusion the result extends immediately to the orbit of $\omega$ under the rotation by $\alpha$ and hence to a dense set of phases), while one is generally primarily interested in the almost sure (which is relative to Lebesgue measure in this case as an irrational rotation is uniquely ergodic with Lebesgue measure as the unique ergodic measure) spectral type. Nevertheless, this is a new phenomenon in the class of one-frequency quasi-periodic potentials and hence may be considered interesting.

The idea behind the proof of Theorem~\ref{t.main4} is the following. Given $\alpha, \varepsilon, \omega$, we place small and very thin bumps centered at the points in the orbit of $\omega$ under $T$. The bumps are so small and thin that when we add them up, we obtain a continuous function of norm less than $\varepsilon$. This can be done in a way that ensures that we have control over the values $V_\omega$ takes. In fact, we can arrange for a sequence that can be analyzed via the Kunz-Souillard method with the net effect that pure point spectrum results for this potential.

Keeping with another theme of the paper, the Kunz-Souillard method allows one to throw in a fixed background potential for free, and hence one can actually prove the following stronger version of Theorem~\ref{t.main4}: Given $\alpha \in \T \setminus \Q$ and $\omega \in \T$, the set of $f \in C(\T,\R)$ for which the Schr\"odinger operator with potential $V(n) = f(\omega + n\alpha)$ has pure point spectrum is dense. To emphasize how little about the base dynamics we use in the proof of this statement, let us generalize the latter statement even further.

\begin{theorem}\label{t.main5}
Suppose $\Omega$ is a compact metric space and $T : \Omega \to \Omega$ is invertible. Assume $\omega \in \Omega$ is such that its orbit $\{ T^n \omega : n \in \Z \}$ is infinite. Then, the set of $f \in C(\Omega,\R)$ for which the Schr\"odinger operator with potential $V_\omega(n) = f(T^n \omega)$ has pure point spectrum is dense.
\end{theorem}

Given information on the rate of return for the infinite orbit in question, one can infer some additional information about the regularity of the sampling functions for which pure point spectrum can be established. For the sake of concreteness, let us discuss this in the context of Theorem~\ref{t.main4}.

Recall that that result is most surprising for Liouville frequencies $\alpha$. However, in that case there is no hope to improve the regularity statement on $f$. Fixing any modulus of continuity, one can show using a Gordon-type argument that for a suitable class of Liouville numbers (that will form a dense $G_\delta$ subset of $\T$), there are no eigenvalues for any $f$ with the given modulus of continuity and any phase $\omega$. Thus one has to look for improved regularity of $f$ only when the frequency $\alpha$ is not Liouville.

Here is a result in this direction.

\begin{theorem}\label{t.main6}
Suppose $\alpha \in \T \setminus \Q$ is Diophantine. Then, given any $\omega \in \T$ and $\gamma\in (0, 1/2)$, the set of H\"older continuous functions $f \in C^\gamma(\T,\R)$ for which the Schr\"odinger operator with potential $V(n) = f(\omega + n\alpha)$ has pure point spectrum is dense in $C^\gamma(\T,\R)$.
\end{theorem}

\begin{remark}
From the proof of Theorem \ref{t.main6} one can see that in fact the Diophantine condition on $\alpha$ can be relaxed. Namely, if $\left\{\frac{p_k}{q_k}\right\}$ is the sequence of continued fraction approximants for $\alpha$, it is enough to require that for some $\kappa>0$ and $\tilde\gamma\in (\gamma, 1/2)$
\begin{equation}\label{e.condkappa}
 \sum_{k=1}^\infty \exp\left(-\kappa q_k^{1-2\tilde\gamma}\right)q_{k+1}^{\tilde\gamma/2}<\infty.
\end{equation}
\end{remark}

It is natural to ask whether one can go beyond the threshold $1/2$ for the H\"older exponent of the sampling function $f$. We suspect this is not the case and ask the following question.

\begin{question}
Suppose $\gamma > 1/2$ and $\alpha \in \T \setminus \Q$. Is it true that for some $\varepsilon > 0$, the Schr\"odinger operator with potential $V(n) = f(\omega + n\alpha)$ has no eigenvalues for every $f \in C^\gamma(\T,\R)$ with $\|f\|_\gamma < \varepsilon$ and every $\omega \in \T$? One can also ask the even stronger question whether there is a uniform choice of $\varepsilon > 0$ that works for every $\alpha \in \T \setminus \Q$.
\end{question}

\subsection{An Extension of the Kunz-Souillard Method}

Let us now discuss the mechanism that underlies the results above. We rely on suitable generalizations of the Kunz-Souillard method.

For a linear functional $\mathcal{L} : \mathbb{R}^N \to \mathbb{R}$, given by
$$
\mathcal{L}(x_1, \ldots, x_N) = b_1 x_1 + \ldots + b_N x_N
$$
with $b_1, \ldots, b_N \in \R$, define $\|\mathcal{L}\|_{+} = \sum_{i=1}^N |b_i|$. Notice that there are many equivalent norms for a given operator $\mathcal{L}$, but we will need to assume a uniform bound on the norms of a family of operators from $\mathbb{R}^N \to \mathbb{R}$ with increasing dimension $N$, and in this case different definitions of norms will lead to essentially different conditions.

Our first extension of the work of Kunz-Souillard \cite{KS80}, indeed of Simon's follow-up work \cite{S82}, reads as follows:

\begin{theorem}\label{t.bs}
Let $\{\xi_n\}_{n=-\infty}^{\infty}$ be independent random variables with distributions of the form $r_n(x) \, dx$, where $r_n(x) = a_n^{-1} r(a^{-1}_n x)$, $a_n > 0$, and $r$ is compactly supported and bounded.

Then, there are constants $d = d(r), \lambda = \lambda(r) > 0$ such that the following holds. Assume the sequence $\{a_n\}_{n = -\infty}^{\infty}$ is bounded and such that
\begin{equation}\label{e.summabilityass}
\sum_{n\in \mathbb{Z}} a_n^{-1/2} e^{- d \sum_{j = 1}^{\lfloor \frac{|n|-1}{2} \rfloor} \min \{ a_{(\mathrm{sgn} \, n) 2j}^2, a_{(\mathrm{sgn} \, n) (2j-1)}^2 , \lambda \} } < \infty.
\end{equation}

%$|a_n|\ge C(1+|n|)^{-\alpha}$ for some $C>0$ and $\alpha\in (0, \frac{1}{2})$.

Let $\{\chi_n\}_{n=-\infty}^{\infty}$ be independent {\rm (}not necessarily identically distributed{\rm )} random variables that are uniformly bounded.

Let $\{ \mathcal{L}_n : \mathbb{R}^{2n-1} \to \mathbb{R} \}_{n = -\infty}^\infty$, $n \in \Z \setminus \{ 0 \}$, be a collection of linear functionals with uniformly bounded $\|\cdot\|_+$ norms.

Then, almost surely, the discrete Schr\"odinger operator with the potential
$$
V:\mathbb{Z}\to \mathbb{R}, \ V(n)=\xi_n+\chi_n+\mathcal{L}_n(\xi_{-|n|+1}, \ldots, \xi_{|n|-1})
$$
has pure point spectrum, where $\mathcal{L}_n(\xi_{-|n|+1}, \ldots, \xi_{|n|-1})$ is set equal to zero for $n = 0$.
\end{theorem}

\begin{remark}
Notice that one is allowed to set all the variables $\{\chi_n\}_{n=-\infty}^{\infty}$ to be constants, and in this way this will produce the statement with a bounded background potential.
\end{remark}

\begin{remark}\label{r.hatr}
One of the natural choices for $r$ is the uniform distribution on the interval $[0,1]$, $r(x) = \chi_{[0,1]}(x)$; this is what was done in \cite{S82}. The value $d$ that corresponds to the distribution $r$ is related to the second derivative of the Fourier transform $\hat r$ at zero (which is always negative for absolutely continuous compactly supported probability distribution).
\end{remark}

\begin{remark}
Notice that any bounded sequence $\{a_n\}$ with $|a_n| \ge C(1+|n|)^{-\alpha}$ for some $C>0$ and $\alpha\in (0, \frac{1}{2})$ satisfies the conditions in Theorem~\ref{t.bs}. This is the condition that was used in \cite{S82}.
\end{remark}

Theorem \ref{t.bs} will provide the tool we use to prove Theorem \ref{t.main5}. In order to prove Theorem~\ref{t.main2}, we will need a more general (and, unfortunately, also more technical) statement.

Suppose we are given a strictly increasing sequence of integers $\{l_m\}_{m\in \mathbb{Z}}$ with $l_{-1} < 0 \le l_0$. This defines a partition of $\mathbb{Z}$ into a union of disjoint intervals $I_m=(l_{m-1}, l_m]$, $m\in \mathbb{Z}$, $\bigcup_{m\in \mathbb{Z}}I_m=\mathbb{Z}$ with $0 \in I_0$. Denote by $m(n)$ the function defined by $n\in I_{m(n)}$.

\begin{theorem}\label{t.general}
Suppose a collection of independent random variables $\{\xi_{n,k}\}_{n\in \mathbb{Z}, k\ge |m(n)|}$ is given, where $\xi_{n,k}$ is distributed with respect to $r_k(x) \, dx$, where $r_k(x)=\varepsilon_k^{-1}r(\varepsilon^{-1}_k x)$, $\varepsilon_k > 0$, and
\begin{equation}\label{e.epsilonsummable}
\sum_{k=0}^\infty \varepsilon_k < \infty.
\end{equation}
Suppose also that a collection of linear functionals $\mathcal{L}_{n,k} : \{ \{ \xi_{s,k}\} : |m(s)| < |m(n)| \} \to \mathbb{R}$ is given, with uniformly bounded $\|\cdot\|_+$ norms.

Then, there is a constant $d = d(r) > 0$ such that the following holds. Suppose that
\begin{equation}\label{e.generalthmsummability}
\sum_{n \in \mathbb{Z}} \varepsilon_{|m(n)|}^{-1/2} e^{- d \sum_{j = 1}^{\lfloor \frac{|n|-1}{2} \rfloor} \min \{ \varepsilon_{|m((\mathrm{sgn} n) 2j)|}^2, \varepsilon_{|m((\mathrm{sgn} n) (2j-1))|}^2 \} } < \infty.
\end{equation}
Let $\{\chi_n\}_{n = -\infty}^{\infty}$ be independent {\rm (}not necessarily identically distributed{\rm )} random variables that are uniformly bounded. Then, almost surely, the discrete Schr\"odinger operator with the potential
$$
V : \mathbb{Z}\to \mathbb{R}, \; V(n) = \sum_{k \ge|m(n)|} \xi_{n,k} + \chi_{n} + \sum_{k=0}^{|m(n)| - 1} \mathcal{L}_{n, k}(\{\xi_{s,k}\}, |m(s)|<|m(n)|)
$$
has pure point spectrum.
\end{theorem}

The organization of the paper is as follows. We work out the generalizations of the Kunz-Souillard method in Section~\ref{s.ksdesc}, where we prove Theorems~\ref{t.bs} and \ref{t.general}. Our applications to limit-periodic potentials are discussed in Section~\ref{sec.lp}, where we prove Theorems~\ref{t.main1}--\ref{t.main2}. Finally, we discuss our applications to quasi-periodic potentials in Section~\ref{sec.qp}, where we prove Theorems~\ref{t.main4}--\ref{t.main6}.

\section{A Generalization of the Kunz-Souillard Method}\label{s.ksdesc}

In this section we prove Theorems~\ref{t.bs} and \ref{t.general}. As was mentioned in the introduction, these theorems are generalizations of the fundamental localization result of Kunz and Souillard from \cite{KS80}. The key additional aspects are the presence of additional random variables and the absence of the independence of the resulting potential values.

Let us begin with Theorem~\ref{t.bs}. Recall the setting of this theorem. The random variables $\{\xi_n\}_{n=-\infty}^{\infty}$ and $\{\chi_n\}_{n=-\infty}^{\infty}$ are independent and the $\{\xi_n\}_{n=-\infty}^{\infty}$ are absolutely continuous with a density $r_n$ that is a rescaling of some fixed bounded and compactly supported density $r$, that is, for each $n \in \Z$, we have $r_n(x)=a_n^{-1}r(a^{-1}_n x)$ for some  $a_n > 0$. The product measure formed from all of these distributions will be denoted by $\mu$, integration with respect to $\mu$ will be denoted by $\mathbb{E} (\cdot)$, and the product space on which $\mu$ is defined will be denoted by $\Upsilon$. Given $\zeta = (\{\xi_n\}_{n=-\infty}^{\infty}, \{\chi_n\}_{n=-\infty}^{\infty}) \in \Upsilon$, we denote the discrete Schr\"odinger operator with potential
\begin{equation}\label{e.thmbspotential}
V(n)=\xi_n + \chi_n + \mathcal{L}_n (\xi_{-|n|+1}, \ldots, \xi_{|n|-1}) , \quad n \in \Z
\end{equation}
by $H_\zeta$. Note that due to the uniform boundedness assumptions, the potential $V$ is bounded.

For $n,m \in \Z$, let us set
$$
a(n,m) = \E \left( \sup_{t\in \R} \left| \left\langle \delta_n , e^{-itH_{\zeta}} \delta_m \right\rangle \right| \right).
$$

Proving decay of $a(n,m)$ in $|n-m|$ establishes some form of dynamical localization. It is well known that dynamical localization implies spectral localization. Concretely, here is the criterion we use to establish pure point spectrum; compare \cite[Theorem~9.21]{CFKS} and \cite[Th\'eor\`eme~IV.4]{KS80}. Because our model is somewhat different, let us give the short argument for the convenience of the reader. We emphasize, however, that the argument does not use any special properties of the model and is the same for the original model and our generalized model and we are merely repeating the argument given in the proof of \cite[Theorem~9.21]{CFKS}.\footnote{Note also that \cite[Theorem~2.1]{S82} claims a similar result with an $\ell^2$-summability assumption. It is not clear whether the statement as formulated there holds as stated. The proof given in \cite{CFKS} and reproduced below needs the $\ell^1$ summability assumption.}

\begin{prop}\label{p.asppcriterion}
Suppose that
$$
\sum_{n \in \Z} |a(n,m)| < \infty
$$
for $m = 0,1$. Then, for $\mu$-almost every $\zeta \in \Upsilon$, the operator $H_\zeta$ has pure point spectrum.
\end{prop}

\begin{proof}
By the RAGE theorem, compare \cite[Theorem~5.8]{CFKS} and \cite[Formula~(9.10)]{CFKS}, we have that
$$
\|\chi_c(H_\zeta) \delta_m\|^2 = \lim_{N \to \infty} \lim_{T \to \infty} \frac{1}{2T} \int_{-T}^T \sum_{|n| \ge N} \left| \left\langle \delta_n , e^{-itH_{\zeta}} \delta_m \right\rangle \right|^2 \, dt,
$$
where $\chi_c(H_\zeta)$ denotes the orthogonal projection onto the continuous subspace associated with $H_\zeta$. Our goal is to show that $\|\chi_c(H_\zeta) \delta_m\|^2$ is zero for $m = 0,1$ and $\mu$-almost every $\zeta \in \Upsilon$.

For $m = 0,1$, we have
\begin{align*}
\mathbb{E}\left( \|\chi_c(H_\zeta) \delta_m\|^2 \right) & = \mathbb{E}\left( \lim_{N \to \infty} \lim_{T \to \infty} \frac{1}{2T} \int_{-T}^T \sum_{|n| \ge N} \left| \left\langle \delta_n , e^{-itH_{\zeta}} \delta_m \right\rangle \right|^2 \, dt \right) \\
& \le \lim_{N \to \infty} \sum_{|n| \ge N} \mathbb{E}\left( \sup_{t \in \R} \left| \left\langle \delta_n , e^{-itH_{\zeta}} \delta_m \right\rangle \right|^2 \right) \\
& \le \lim_{N \to \infty} \sum_{|n| \ge N} \mathbb{E}\left( \sup_{t \in \R} \left| \left\langle \delta_n , e^{-itH_{\zeta}} \delta_m \right\rangle \right| \right) \\
& = \lim_{N \to \infty} \sum_{|n| \ge N} |a(m,n)| \\
& = 0,
\end{align*}
where we used $\left| \left\langle \delta_n , e^{-itH_{\zeta}} \delta_m \right\rangle \right| \le 1$ in the third step and our assumption in the last step.

Since $\|\chi_c(H_\zeta) \delta_m\|^2 \ge 0$, it follows that $\|\chi_c(H_\zeta) \delta_m\|^2$ is zero for $m = 0,1$ and $\mu$-almost every $\zeta \in \Upsilon$, concluding the proof by cyclicity of $\{ \delta_0 , \delta_1 \}$.
\end{proof}

The next step is to restrict the operators $H_\zeta$ to finite intervals. For $L\in\Z_+$, we denote by $H_{\zeta}^{(L)}$ the restriction of $H_{\zeta}$ to $\ell^2(\{-L,\ldots, L\})$ with Dirichlet boundary conditions.

For $|n|, |m| \le L$, we define $a_L(n,m)$ to be $a(n,m)$ with $H_{\zeta}$ replaced by $H_{\zeta}^{(L)}$, that is,
$$
a_L(n,m) = \E \left( \sup_{t\in \R} \left| \left\langle \delta_n , e^{-itH_{\zeta}^{(L)}} \delta_m \right\rangle \right| \right).
$$
It is easy to see that $H_{\zeta}^{(L)}$ has $2L+1$ real simple eigenvalues
$$
E_{\zeta}^{L,1} < E_{\zeta}^{L,2} < \cdots < E_{\zeta}^{L,2L+1}.
$$
For each $k$, let $\varphi_{\zeta}^{L,k}$ denote a normalized eigenvector corresponding to $E_{\zeta}^{L,k}$.  We now define
\begin{equation}\label{e.rhoLdef}
\rho_L(n,m) = \E \left( \sum_{k=1}^{2L+1} \left| \left\langle \delta_n, \varphi_{\zeta}^{L,k} \right\rangle \right| \left| \left\langle \delta_m , \varphi_{\zeta}^{L,k} \right\rangle \right| \right).
\end{equation}

The statements in the following lemma are easy to prove; compare the discussion in \cite[pp.~192--193]{CFKS}. Again, for the sake of completeness and the convenience of the reader, we reproduce the short arguments.

\begin{lemma} \label{l.ks.est1}
{\rm (a)} For $n,m \in \Z$, we have
\begin{equation}\label{e.ksest1}
a(n,m) \leq \liminf_{L \to \infty} a_L(n,m).
\end{equation}
{\rm (b)} If $|n|,|m| \leq L$, then
\begin{equation}\label{e.ksest2}
a_L(n,m) \leq \rho_L(n,m).
\end{equation}
\end{lemma}

\begin{proof}
(a) Regard $H_\zeta^{(L)}$ as an operator on $\ell^2(\Z)$ (i.e., set $\langle \delta_j ,  H_\zeta^{(L)} \delta_k \rangle = 0$ if $|j| > L$ or $|k| > L$). Clearly, $H_\zeta^{(L)}$ converges strongly to $H_\zeta$ as $L \to \infty$. Thus, $e^{-itH_\zeta^{(L)}}$ converges strongly to $e^{-itH_\zeta}$ as $L \to \infty$. Consequently, the assertion follows by Fatou's Lemma.

(b) We have
\begin{align*}
a_L(n,m) & = \E \left( \sup_{t \in \R} \left| \left\langle \delta_n , e^{-itH_\zeta^{(L)}} \delta_m \right\rangle \right| \right) \\
& = \E \left( \sup_{t \in \R} \Big| \Big\langle \delta_n , e^{-itH_\zeta^{(L)}} \sum_k \left\langle \varphi_\zeta^{L,k} , \delta_m \right\rangle \varphi_\zeta^{L,k} \Big\rangle \Big| \right) \\
& = \E \left( \sup_{t \in \R} \Big|  \sum_k e^{-itE_{\zeta}^{L,k}} \left\langle \delta_n , \varphi_\zeta^{L,k} \right\rangle \left\langle \delta_m , \varphi_\zeta^{L,k} \right\rangle \Big| \right) \\
& \le \E \left( \sum_k \left| \left\langle \delta_n , \varphi_\zeta^{L,k} \right\rangle \right| \, \left| \left\langle \delta_m , \varphi_\zeta^{L,k} \right\rangle \right| \right) \\
& = \rho_L(n,m),
\end{align*}
as claimed.
\end{proof}

Thus one seeks upper bounds for $\rho_L(n,m)$ in terms of $|n - m|$ that are uniform in $L$ so that they survive as $L \to \infty$. Let us consider the case $m = 0$ and $n > 0$ explicitly; the other cases are similar. The quantity $\rho_L(n,0)$ will be estimated in terms of suitable integral operators which arise through a certain change of variables.

The expectation $\E(\cdot)$ in the definition of $\rho_L(n,m)$ is given by integration over all the random variables $\{\xi_j\}_{j=-\infty}^{\infty}$ and $\{\chi_j\}_{j=-\infty}^{\infty}$. Two important remarks are in order:

\medskip

(i) We will freeze all the $\chi_j$'s and regard them as a background potential. The estimates we prove will be uniform in the choice of $\chi = \{\chi_j\}_{j=-\infty}^{\infty}$ due to the uniform boundedness of these random variables. Thus, by Fubini the estimates we prove for frozen $\chi_j$'s extend to estimates for $\rho_L(n,m)$.

(ii) Freezing the $\chi_j$'s, the operator $H_\zeta^{(L)}$ does not depend on all of the $\xi_j$'s, but rather only on $\{\xi_j\}_{j=-L}^{L}$. Thus, the inner integral we need to consider (for frozen $\chi_j$'s) is a $2L+1$-fold iterated integral. In other words, we consider
$$
\rho_L(n,m;\chi) = \int \cdots \int \left( \sum_{k=1}^{2L+1} \left| \left\langle \delta_n, \varphi_{\zeta}^{L,k} \right\rangle \right| \left| \left\langle \delta_m , \varphi_{\zeta}^{L,k} \right\rangle \right| \right) \prod_{j = -L}^L r_j(v_j) \, dv_{-L} \cdots dv_L,
$$
where $\zeta = (\chi,\xi)$, and $\xi_j = v_j$ for $|j| \le L$. We will apply a suitable change of variables to this integral.

\medskip

Fix $E \in \R$, and define
\begin{align*}
\left( U_0 f\right)(x) & = |x|^{-1} f \left( |x|^{-1} \right), \\
\left( S_E^{(j)} f \right)(x) & =  \int_{\R} \! r_j \left( E - x - y^{-1} \right) f(y) \, dy, \\
\left( T_E^{(j)} f \right)(x) & =  \int_{\R} \! r_j \left( E - x - y^{-1} \right) |y|^{-1} f(y) \, dy.
\end{align*}
By our uniform boundedness assumption, there exists a compact interval $\Sigma_0$ that contains all spectra $\sigma(H_\zeta)$ and all eigenvalues $E_{\zeta}^{L,k}$.

\begin{lemma}\label{l.ks1}
Fix $L \in \Z_+$ and $n$ with $0< n \leq L$. Set $\phi_r(x) = r_L(E - \chi_L - x)$ and $\phi_\ell(x) = r_{-L}(E - \chi_{-L} - x)$. Then, we have
\begin{equation}\label{e.ksest3}
\rho_L(n,0; \chi)  = \int_{\Sigma_0} \left\langle T_{E - \chi_1}^{(1)} \cdots T_{E - \chi_{n-1}}^{(n-1)} S_{E - \chi_{n}}^{(n)} \cdots S_{E - \chi_{L-1}}^{(L-1)} \phi_r , U S_{E - \chi_{0}}^{(0)} \cdots S_{E - \chi_{-L+1}}^{(-L+1)} \phi_\ell \right\rangle_{L^2(\R)} \, dE.
\end{equation}
\end{lemma}

\begin{proof}
As pointed out earlier, $\rho_L(m,0)$ can be expressed by a $(2L+1)$-fold integral:
\begin{equation}\label{e.ks8}
\rho_L(n,0;\chi) = \int \cdots \int \left( \sum_{k=1}^{2L+1} \left| \left\langle \delta_n, \varphi_{\tilde V}^{L,k} \right\rangle \right| \left| \left\langle \delta_0 , \varphi_{\tilde V}^{L,k} \right\rangle \right| \right) \prod_{j = -L}^L r_j(v_j) \, dv_{-L} \cdots dv_L,
\end{equation}
where
$$
\tilde V = \begin{pmatrix} v_{-L} + \chi_{-L} + \mathcal{L}_{-L}(v_{-L+1}, \ldots, v_{L-1}) \\ \vdots \\ v_{-1} + \chi_{-1} +\mathcal{L}_{-1}(v_0) \\ v_0 + \chi_0 \\ v_1 + \chi_1 + \mathcal{L}_1(v_0) \\ \vdots \\ v_L + \chi_{L} + \mathcal{L}_L(v_{-L+1}, \ldots, v_{L-1}) \end{pmatrix}.
$$

To make the notation more manageable, from now on we will omit the arguments of the linear functionals $\mathcal{L}_n$.  Denote the eigenvalues (listed in increasing order) and corresponding normalized eigenvectors of
$$
\begin{pmatrix} \tilde V_{-L} & 1 & &&& \\ 1 & \tilde V_{-L+1} & 1 &&& \\ & \ddots & \ddots & \ddots &&  \\ && \ddots & \ddots & \ddots & \\ &&& 1 & \tilde V_{L-1} & 1 \\ &&&& 1 & \tilde V_{L} \end{pmatrix}
$$
by $\{ E_{\tilde V}^{L,k} \}_{-L \le k \le L}$ and $\{ \varphi_{\tilde V}^{L,k} \}_{-L \le k \le L}$, respectively.

Thus, if $E$ is $E_{\tilde V}^{L,k}$ and $u$ is $\varphi_{\tilde V}^{L,k}$, then we have
\begin{equation}\label{e.ks5}
u(n+1) + u(n-1) + (v_n + \chi_n + \mathcal{L}_n) u(n) = E u(n)
\end{equation}
for $-L \le n \le L$, where $u(-L-1) = u(L+1) = 0$. We rewrite \eqref{e.ks5} as
\begin{equation}\label{e.ks6}
v_n = E - \chi_n - \frac{u(n+1)}{u(n)} - \frac{u(n-1)}{u(n)}-\mathcal{L}_n.
\end{equation}
This motivates a change of variables,
\begin{equation}\label{e.ks7}
\{ v_n \}_{n = -L}^L \quad \longleftrightarrow \quad \{ x_{-L} , \ldots , x_{-1} , E , x_1 , \ldots , x_L \},
\end{equation}
where
$$
E = E_{\tilde V}^{L,k}
$$
and
$$
x_n = \begin{cases} \frac{\varphi_{\tilde V}^{L,k} (n+1)}{\varphi_{\tilde V}^{L,k} (n)} & n < 0 \\ \frac{\varphi_{\tilde V}^{L,k} (n-1)}{\varphi_{\tilde V}^{L,k} (n)} & n > 0, \end{cases}
$$
so that
$$
v_n = \begin{cases} E - \chi_n - x_{n-1}^{-1} - x_n -\mathcal{L}_n& n < 0 \\ E - \chi_n - x_{-1}^{-1} - x_1^{-1} & n = 0 \\ E -\chi_n - x_{n+1}^{-1} - x_n-\mathcal{L}_n & n > 0, \end{cases}
$$
with $x_{-L-1}^{-1} = x_{L+1}^{-1} = 0$ (which is natural in view of the definition above).

Notice that this change of variables can be represented as a composition $A\circ F$, where
$$
F(x, E)=\begin{cases} E - \chi_n - x_{n-1}^{-1} - x_n & n < 0, \\ E - \chi_n - x_{-1}^{-1} - x_1^{-1} & n = 0, \\ E - \chi_n - x_{n+1}^{-1} - x_n & n > 0, \end{cases}
$$
and $A : \mathbb{R}^{2L+1} \to \mathbb{R}^{2L+1}$ is the linear map given by
$$
\begin{pmatrix}
  W_{-L} \\
 \vdots \\
 W_0 \\
 \vdots \\
W_L \\
\end{pmatrix}
\to
\begin{pmatrix}
  W_{-L} - \mathcal{L}_{-L}(W_{-L}, \ldots, W_0, \ldots, W_L)\\
 \vdots \\
 W_0 \\
 \vdots \\
W_L - \mathcal{L}_{L}(W_{-L}, \ldots, W_0, \ldots, W_L) \\
\end{pmatrix}
$$
Notice that $\det A=1$.

We wish to rewrite \eqref{e.ks8} using this change of variables. In order to do this, we have to determine the Jacobian of the change of variables \eqref{e.ks7}. Since $\det A=1$, it is equal to the Jacobian of $F$. We claim that it satisfies
\begin{align}\label{e.ks9}
\det J & = \varphi_{\tilde V}^{L,k} (0)^{-2}.
\end{align}

To prove \eqref{e.ks9}, we note that
$$
J = \begin{pmatrix} -1 & x_{-L}^{-2} &&&&&&& \\ & -1 & x_{-L+1}^{-2} &&&&&& \\ && \ddots & \ddots &&&&& \\ &&& -1 & x_{-1}^{-2} &&&& \\ 1 & 1 & \cdots & 1 & 1 & 1 & \cdots & 1 & 1 \\ &&&& x_1^{-2} & - 1 &&& \\ &&&&& x_2^{-2} & -1 && \\ &&&&&& \ddots & \ddots & \\ &&&&&&& x_{L}^{-2} & -1 \end{pmatrix}
$$
where the row index runs from $x_{-L}$ at the top to $x_L$ at the bottom (with $E$ in the middle) and the column index runs from $v_{-L}$ on the left to $v_L$ on the right. Since this is the same matrix that arises in the standard Kunz-Souillard method, one can now proceed as in that case. We refer the reader to \cite{CFKS, DF16, KS80} for details, but note that iterated expansion by minors yields that
\begin{align*}
\det J & = 1 + x_1^{-2} \left\{ 1 + x_2^{-2} \left\{ 1 + \cdots x_{L-1}^{-2} \left\{ 1 + x_L^{-2} \right\} \cdots \right\} \right\} \\
& \qquad  + x_{-1}^{-2} \left\{ 1 + x_{-2}^{-2} \left\{ 1 + \cdots x_{-L+1}^{-2} \left\{ 1 + x_{-L}^{-2} \right\} \cdots \right\} \right\} \\
& = \frac{\varphi_{\tilde V}^{L,k} (0)^2}{\varphi_{\tilde V}^{L,k} (0)^2} + \sum_{n=-1}^{-L} \frac{\varphi_{\tilde V}^{L,k} (n)^2}{\varphi_{\tilde V}^{L,k} (0)^2} + \sum_{n=1}^{L} \frac{\varphi_{\tilde V}^{L,k} (n)^2}{\varphi_{\tilde V}^{L,k} (0)^2} \\
& = \frac{\|\varphi_{\tilde V}^{L,k}\|^2}{\varphi_{\tilde V}^{L,k} (0)^2} \\
& = \varphi_{\tilde V}^{L,k} (0)^{-2},
\end{align*}
since $\varphi_{\tilde V}^{L,k}$ is normalized. This proves \eqref{e.ks9}.

We also note that
\begin{equation}\label{e.ks10}
\left| \varphi_{\tilde V}^{L,k} (n) \right| \left| \varphi_{\tilde V}^{L,k} (0) \right|^{-1} = \left| x_1^{-1} \cdots x_n^{-1} \right|.
\end{equation}

We are now ready to carry out the substitution \eqref{e.ks7} in the formula \eqref{e.ks8} for $\rho_L(n,0;\chi)$, using the identities \eqref{e.ks9} and \eqref{e.ks10}. For $0 \le k \le 2L$, let
$$
R_k = \{ (x,E) : x \in \R^{2L}, E \in \Sigma_0, \, \text{precisely $k$ components of $x$ are negative} \}.
$$
\begin{align*}
\rho_L(n,0;\chi) & = \int \cdots \int \left( \sum_k \left| \left\langle \delta_n , \varphi_{\tilde V}^{L,k} \right\rangle \right| \, \left| \left\langle \delta_0 , \varphi_{\tilde V}^{L,k} \right\rangle \right| \right) \prod_{j = -L}^L r_j(v_j) \, dv_{-L} \cdots dv_L \\
& = \sum_{k = -L}^L \int \cdots \int \left| \varphi_{\tilde V}^{L,k}(n) \right| \, \left| \varphi_{\tilde V}^{L,k} (0) \right| \prod_{j = -L}^L r_j(v_j) \, dv_{-L} \cdots dv_L \\
& = \sum_{k = -L}^L \int \cdots \int \left| \varphi_{\tilde V}^{L,k}(n) \right| \, \left| \varphi_{\tilde V}^{L,k} (0) \right|^{-1} \prod_{j = -L}^L r_j(v_j) \left| \varphi_{\tilde V}^{L,k} (0) \right|^2 \, dv_{-L} \cdots dv_L \\
& = \sum_{k = 0}^{2L} \int_{R_k} \left| x_1^{-1} \cdots x_n^{-1} \right| \left( \prod_{j=-1}^{-L} r_j (E - \chi_j - x_{j-1}^{-1} - x_j) \right) r_0(E - x_1^{-1} - x_{-1}^{-1}) \\
& \qquad \times \left( \prod_{j=1}^{L} r_j (E - \chi_j - x_{j+1}^{-1} - x_j) \right) \, dx_{-L} \cdots dx_{-1} \, dx_1 \cdots dx_L \, dE \\
& = \int_{\Sigma_0} \int_{\R^{2L}} \left| x_1^{-1} \cdots x_n^{-1} \right| \left( \prod_{j=-1}^{-L} r_j (E - \chi_j - x_{j-1}^{-1} - x_j) \right) r_0(E - x_1^{-1} - x_{-1}^{-1}) \\
& \qquad \times \left( \prod_{j=1}^{L} r_j (E - \chi_j - x_{j+1}^{-1} - x_j) \right) \, dx_{-L} \cdots dx_{-1} \, dx_1 \cdots dx_L \, dE.
\end{align*}

Here, the fourth step follows by oscillation theory; see \cite{DF16} for a detailed explanation of this argument.

For fixed $E \in \Sigma_0$, it follows from the definitions that the inner integral has the required form:
\begin{align*}
& \left\langle T_{E - \chi_1}^{(1)} \cdots T_{E - \chi_{n-1}}^{(n-1)} S_{E - \chi_{n}}^{(n)} \cdots S_{E - \chi_{L-1}}^{(L-1)} \phi_r , U S_{E - \chi_{0}}^{(0)} \cdots S_{E - \chi_{-L+1}}^{(-L+1)} \phi_\ell \right\rangle_{L^2(\R)} \\
& \quad = \int_{\R^{2L}} \left| x_1^{-1} \cdots x_n^{-1} \right| \left( \prod_{j=-1}^{-L} r_j (E - \chi_j - x_{j-1}^{-1} - x_j) \right) r_0(E - x_1^{-1} - x_{-1}^{-1}) \\
& \qquad \times \left( \prod_{j=1}^{L} r_j (E - \chi_j - x_{j+1}^{-1} - x_j) \right) \, dx_{-L} \cdots dx_{-1} \, dx_1 \cdots dx_L
\end{align*}
The formula for $\rho_L(n,0)$ claimed in the lemma therefore follows.
\end{proof}

Given this representation, we are naturally interested in estimates for the integral operators $S^{(\cdot)}_{\cdot}$ and $T^{(\cdot)}_{\cdot}$. In the following lemma, for which we refer the reader to \cite{S82},  we denote the norm of an operator $T : L^p(\R) \to L^q(\R)$ by $\|T\|_{p,q}$.

\begin{lemma}\label{l.ks4}
{\rm (a)} $\sup \{ \|S_{E - \chi_{j}}^{(j)}\|_{1,1} : E \in \Sigma_0, \; j \in \Z, \; (\xi,\chi) \in \mathrm{supp} \, \mu \} \le 1$. \\
{\rm (b)} $\sup \{ \|S_{E - \chi_{j}}^{(j)}\|_{1,2} : E \in \Sigma_0, \; (\xi,\chi) \in \mathrm{supp} \, \mu \} \le \|r_j\|_\infty^{1/2} = a_j^{-1/2} \|r\|_\infty^{1/2}$. \\
{\rm (c)} $\sup \{ \|T_{E - \chi_{j}}^{(j)}\|_{2,2} : E \in \Sigma_0, \; j \in \Z, \; (\xi,\chi) \in \mathrm{supp} \, \mu \} \le 1$. \\
{\rm (d)} For $\lambda \ge 0$ sufficiently small, $\hat r$, the Fourier transform of $r$, obeys
$$
\sup_{|k| \ge \lambda} |\hat r(k)|^2 \le e^{- c |\lambda|^2}
$$
for a suitable $c = c(r) > 0$. \\
{\rm (e)} There exist $K_0 = K_0(r), \lambda = \lambda(r), c = c(r) > 0$ such that for every $j$, we have
\begin{align*}
\|T_{E - \chi_j}^{(j)} T_{E - \chi_{j-1}}^{(j-1)}\|_{2,2} & \le \left( \frac{15}{16} + \frac{1}{16} \sup_{|k| \ge K_0 \min\{a_{j}, a_{j-1}\}} |\hat r(k)|^2 \right)^{1/2} \\
& \le e^{- c K_0^2 \min\{a_{j}^2, a_{j-1}^2, \lambda \}}.
\end{align*}
\end{lemma}

\begin{proof}[Proof of Theorem~\ref{t.bs}.]
By Proposition~\ref{p.asppcriterion} it suffices to show that
\begin{equation}\label{e.summability}
\sum_{n \in \Z} |a(n,m)| < \infty
\end{equation}
for $m = 0,1$. Consider the case $m = 0$ and $n > 0$. Let $k = \lfloor \frac{n-1}{2} \rfloor$. Then it follows from our considerations above that
\begin{align*}
|a(n,0)| & \le \mathrm{Leb}(\Sigma_0) \liminf_{L \to\infty} \sup_{E \in \Sigma_0, \, (\xi,\chi) \in \mathrm{supp} \, \mu} \Big| \Big\langle T_{E - \chi_1}^{(1)} \cdots T_{E - \chi_{n-1}}^{(n-1)} S_{E - \chi_{n}}^{(n)} \cdots S_{E - \chi_{L-1}}^{(L-1)} \phi_r , \\
& \qquad \qquad \qquad \qquad \qquad \qquad \qquad \qquad U S_{E - \chi_{0}}^{(0)} \cdots S_{E - \chi_{-L+1}}^{(-L+1)} \phi_\ell \Big\rangle_{L^2(\R)} \Big| \\
& \le \mathrm{Leb}(\Sigma_0) \, a_n^{-1/2} \|r\|_\infty^{1/2} e^{-c K_0^2 \sum_{j = 1}^k \min \{ a_{2j}^2, a_{2j-1}^2, \lambda \} } a_0^{-1/2} \|r\|_\infty^{1/2}.
\end{align*}
Here we used that $\|\phi_r\|_1 = \|\phi_\ell\|_1 = 1$ and that $U : L^2(\R) \to L^2(\R)$ is unitary, along with the estimates in Lemma~\ref{l.ks4}. The constant $\lambda = \lambda(r) > 0$ is chosen small enough so that we can apply the estimate from Lemma~\ref{l.ks4}.(e).

The estimate for $|a(n,m)|$ when $n < 0$ or $m = 1$ follows in a similar way. By our assumption \eqref{e.summabilityass} (with $d = c K_0^2$), the summability statement \eqref{e.summability} follows.
\end{proof}

\begin{proof}[Proof of Theorem~\ref{t.general}.]
Again we freeze the random variables $\{ \chi_n \}_{n \in \Z}$ and treat them as a background potential. But we will freeze additional variables as well. Recall that the potential is of the form
\begin{align*}
V(n) & = \sum_{k \ge |m(n)|} \xi_{n,k} + \chi_{n} + \sum_{k=0}^{|m(n)| - 1} \mathcal{L}_{n, k}(\{\xi_{s,k}\}, |m(s)|<|m(n)|) \\
& = \xi_{n,|m(n)|} + \sum_{k > |m(n)|} \xi_{n,k} + \chi_{n} + \sum_{k=0}^{|m(n)| - 1} \mathcal{L}_{n, k}(\{\xi_{s,k}\}, |m(s)|<|m(n)|)
\end{align*}
The background potential will be given by
$$
\sum_{k > |m(n)|} \xi_{n,k} + \chi_{n},
$$
while the primary random variables are $\xi_{n,|m(n)|}$. As above, we have the additional term
$$
\sum_{k=0}^{|m(n)| - 1} \mathcal{L}_{n, k}(\{\xi_{s,k}\}, |m(s)|<|m(n)|)
$$
that introduces some correlations between sites.

As above we consider the quantity $\rho_L(n,m)$, which is defined in terms of the finite-volume operators $H_\zeta^{(L)}$ as in \eqref{e.rhoLdef}. The analogs of Proposition~\ref{p.asppcriterion} and Lemma~\ref{l.ks.est1} are proved in the same way.

The quantity $\rho_L(n,m)$ is again viewed as an integral over the variables we freeze and we wish to prove estimates for the integrand, which may now be denoted by $\rho_L(n,m; \{ \xi_{\cdot,k} \}_{k > |m(\cdot)|}, \chi)$, that are uniform in the frozen variables and hence extend to estimates for $\rho_L(n,m)$.

The quantity $\rho_L(n,m; \{ \xi_{\cdot,k} \}_{k > |m(\cdot)|}, \chi)$ is studied in the same way as before. The change of variables used above transforms it (for, say, $m = 0$ and $n > 0$) into an expression like \eqref{e.ksest3}, which can again be estimated with the help of Lemma~\ref{l.ks4}. Using the assumptions of the present theorem, we obtain in this way the estimate
$$
|a(n,0)| \le \mathrm{Leb}(\Sigma_0) \varepsilon_{|m(n)|}^{-1/2} \|r\|_\infty^{1/2} e^{-c K_0^2 \sum_{j = 1}^{\lfloor \frac{n-1}{2} \rfloor} \min \{ \varepsilon_{|m(2j)|}^2, \varepsilon_{|m(2j-1)|}^2, \lambda \} } \varepsilon_0^{-1/2} \|r\|_\infty^{1/2}.
$$
A similar estimate can be shown for $|a(n,0)|$ when $n < 0$, and also for $|a(n,1)|$.

By our assumptions \eqref{e.epsilonsummable}, we have $\varepsilon_k \to 0$ as $k \to \infty$, and hence
$$
\min \{ \varepsilon_{|m((\mathrm{sgn} n) 2j)|}^2, \varepsilon_{|m((\mathrm{sgn} n) (2j-1))|}^2, \lambda \} = \min \{ \varepsilon_{|m((\mathrm{sgn} n) 2j)|}^2, \varepsilon_{|m((\mathrm{sgn} n) (2j-1))|}^2 \}
$$
for $|j|$ sufficiently large. Therefore, by \eqref{e.generalthmsummability} it follows that these upper bounds are summable and hence Proposition~\ref{p.asppcriterion} is applicable. The result follows.
\end{proof}

\section{Applications to Limit-Periodic Potentials}\label{sec.lp}

In this section we prove Theorems~\ref{t.main1}--\ref{t.main2}. It is clear that Theorem~\ref{t.main1} follows immediately from Theorem~\ref{t.main2}, so let us discuss the proof of Theorem~\ref{t.main2}.

First let us formulate a simple statement that is to be used to choose the parameters of the construction.

\begin{prop}\label{p.sequences1}
Given constants $\varepsilon > 0$ and $\gamma > 0$, there is a decreasing  sequence $\{\varepsilon_k\}_{k \ge 0}$ of positive numbers and an increasing sequence $\{n_k\}_{k \ge 0}$ of natural numbers such that
\begin{enumerate}
\item $\sum_{k=0}^\infty \varepsilon_k < \varepsilon$;
\item $2n_k+1$ is a non-trivial multiple of $2n_{k-1}+1$ for each $k\ge 1$;
\item $\sum_{k=2}^\infty \varepsilon_k^{-5/2}e^{-\gamma n_{k-1}\varepsilon_k^2}<\infty$.
%\item $\sum_{s=k}^\infty\varepsilon_s\le \delta\left(\frac{1}{k}, n_{k-2}\right)$ for any $k\ge 3$.
\end{enumerate}
\end{prop}

\begin{proof}[Proof of Proposition \ref{p.sequences1}]
 Choose any decreasing sequence $\{\varepsilon_k\}_{k\ge 0}$ with $\sum_{k=0}^\infty \varepsilon_k < \varepsilon$. One can choose a sequence $\{n_k\}_{k\ge 0}$ in such a way that $2n_{k}+1$ is a multiple of $2n_{k-1}+1$, and such that
 $\varepsilon_k^{-5/2}e^{-\gamma n_{k-1}\varepsilon_k^2}<\frac{1}{2^k}.$
Now all the required properties must be satisfied.
\end{proof}

\begin{proof}[Proof of Theorem~\ref{t.main2}.]
Suppose a bounded $V : \Z \to \R$ and $\varepsilon > 0$ are given. Denote the uniform distribution on $[0,1]$ by $r(x) \, dx$ (i.e., $r(x) = \chi_{[0,1]}(x)$). Let  $\gamma$ be determined by $\gamma=\frac{d(r)}{2}$. %associated with this $r$ via TBA.

%With the $M$ just defined, Lemma~\ref{l.cont} defines a function $\delta:\mathbb{R}^+\times \mathbb{N}\to \mathbb{R}^+$, sending $(\chi,N)$ to $\delta(\chi,N)$. Then,

With $\varepsilon$ and $\gamma$ as above, Proposition~\ref{p.sequences1} yields sequences $\{\varepsilon_k\}_{k \ge 0}$ and $\{n_k\}_{k \ge 0}$ with the properties listed there. Define the intervals $I_m$, $m \in \Z$ by
$$
I_0 = [-n_0,n_0], \quad I_m = -I_{-m} = (n_{m-1} , n_m] , \, m \ge 1.
$$
As in the paragraph preceding the statement of Theorem~\ref{t.general}, this defines a map $n \mapsto m(n)$ via $n \in I_{m(n)}$.

Set $r_k(x) = \varepsilon_k^{-1}r(\varepsilon^{-1}_k x)$, $k \ge 0$, and consider independent random variables $\{\xi_{n,k}\}_{n\in \mathbb{Z}, k \ge |m(n)|}$, where $\xi_{n,k}$ is distributed with respect to $r_k(x) \, dx$.

The functionals $\mathcal{L}_{n,k}:\{ (\xi_{s,k})_{|m(s)|<|m(n)|} \}\to \mathbb{R}$ are defined by
$$
\mathcal{L}_{n,k} \left( ( \xi_{s,k} )_{|m(s)| < |m(n)|} \right) = \xi_{(n + n_k) \!\!\!\! \mod (2n_k+1) - n_k}
$$
and the random variables $\chi_n$ are given by
$$
\chi_n = V(n),
$$
that is, they are actually non-random in our scenario (which is permitted in Theorem~\ref{t.general}).

Then,
$V_\mathrm{lp}$ defined by
$$
V_\mathrm{lp}(n) = \sum_{k \ge|m(n)|} \xi_{n,k} + \sum_{k=0}^{|m(n)|-1} \mathcal{L}_{n, k}((\xi_{s,k})_{|m(s)|<|m(n)|})
$$
is limit-periodic and obeys $\|V_\mathrm{lp}\|_\infty < \varepsilon$.

We now apply Theorem~\ref{t.general} with the choices specified above. Since $\{\varepsilon_k\}$ is decreasing, and due to the symmetry of the constructed perturbation,  to apply Theorem~\ref{t.general} we need to check that
\begin{equation}\label{e.condi}
\sum_{n\ge 0}\varepsilon_{m(n)}^{-1/2}e^{-d\sum_{j=1}^{{\lfloor \frac{n-1}{2} \rfloor}}\varepsilon^2_{m(2j)}}<\infty.
\end{equation}
We have
\begin{align*}
\sum_{n = n_{k-1}+1}^{n_k} \varepsilon_{m(n)}^{-1/2} e^{-d\sum_{j=1}^{{\lfloor \frac{n-1}{2} \rfloor}} \varepsilon^2_{m(2j)}} & \le C \varepsilon_k^{-1/2} \sum_{n=n_{k-1}+1}^{n_k} e^{-d\frac{n}{2}\varepsilon_k^2} \\
& < C \varepsilon_k^{-1/2} \frac{e^{-d\frac{n_{k-1}}{2} \varepsilon_k^2}}{1-e^{-\frac{d}{2}\varepsilon_k^2}} \\
& < C' \varepsilon_k^{-5/2} e^{-\gamma{n_{k-1}}\varepsilon_k^2}.
\end{align*}
Since $\sum_{k=2}^\infty \varepsilon_k^{-5/2}e^{-\gamma n_{k-1}\varepsilon_k^2}<\infty$, the condition \eqref{e.condi} follows. Thus, almost surely, the Schr\"odinger operator with potential
$$
V(n) + V_\mathrm{lp}(n) = \sum_{k \ge|m(n)|} \xi_{n,k} + \chi_n + \sum_{k=0}^{|m(n)|-1} \mathcal{L}_{n, k}((\xi_{s,k})_{|m(s)|<|m(n)|})
$$
has pure point spectrum, and this completes the proof of Theorem~\ref{t.main2}.
\end{proof}

\section{Applications to Quasi-Periodic Potentials}\label{sec.qp}

In this section we prove Theorems~\ref{t.main4}, \ref{t.main5}, and \ref{t.main6}. Obviously, Theorem~\ref{t.main4} is a special case of Theorem~\ref{t.main5}, so let us prove the latter theorem.

\begin{proof}[Proof of Theorem~\ref{t.main5}.]
Suppose $\Omega$ is a compact metric space and $T : \Omega \to \Omega$ is invertible and has an infinite orbit. Given a continuous $f:\Omega\to \mathbb{R}$, set $\chi_n=f(T^n(\omega))$ (formally speaking, $\chi_n$ is the Dirac mass located at $f(T^n(\omega))$). Fix an arbitrarily small $\varepsilon > 0$. We need to find a continuous $g : \Omega \to \mathbb{R}$ such that $\|g - f\|_{\infty} < \varepsilon$ and such that the discrete Schr\"odinger operator with potential $V(n) = g(T^n(\omega))$ has pure point spectrum. This will imply Theorem \ref{t.main5}.

Take a sequence $\{a_n\}_{n\in \mathbb{Z}}$ of the form $a_n=\frac{\varepsilon}{100}(1+|n|)^\alpha$ with some  $\alpha\in (0,1/2)$. Let us partition $\mathbb{Z}$ into intervals $[n_{i-1}+1, n_i]$, $i \in \mathbb{Z}$ in such a way that for any sequence of indices $m_i \in \mathbb{Z}$ with $m_i \in [n_{i-1}+1, n_{i}]$, we have $\sum_{i \in \mathbb{Z}} a_{m_i} < \varepsilon$.

Choose neighborhoods $\{B_n(T^n\omega)\}_{n\in \mathbb{Z}}$ in such a way that for each $i \in \mathbb{Z}$, the neighborhoods $B_{n_{i-1}+1}, \ldots, B_{n_{i}}$ are disjoint, and besides $T^j(\omega) \not \in B_n(T^n\omega)$ for any $|j|<|n|$. Also, choose a sequence of continuous functions $g_n:\Omega\to \mathbb{R}, n\in \mathbb{Z},$ in such a way that $g_n\ge 0$, $g_n|_{\Omega\backslash B_n}\equiv 0$, and $\max g_n=g_n(T^n\omega)=a_n$. Notice that for any sequence $\{\mu_n\}_{n \in \mathbb{Z}}$ with $\mu \in [0,1]$, the series $\sum_{n \in \mathbb{Z}} \mu_n g_n$ converges uniformly and hence defines a continuous real-valued function on $\Omega$ of norm not greater than $\varepsilon$.

Let $\{\xi_n\}_{n\in \mathbb{Z}}$ be a family of independent random variables that are uniformly distributed on $[0,1]$. Consider the random function $g : \Omega \to \mathbb{R}$, $g = f + \sum_{n \in \mathbb{Z}} \xi_n g_n$. A direct application of Theorem~\ref{t.bs} implies that almost surely the discrete Schr\"odinger operator with potential $V(n) = g(T^n\omega)$ has pure point spectrum.
\end{proof}

Let us now prove Theorem~\ref{t.main6}.

\begin{proof}[Proof of Theorem~\ref{t.main6}.]
Without loss of generality we can take $\omega=0\in \mathbb{T}$. Consider the orbit $\{z_n\}$, $z_n=n\alpha\, (\text{mod}\ 1)$. Well-known properties of the continued fraction approximants $\left\{\frac{p_k}{q_k}\right\}$ imply that for any $n_1, n_2$ with $|2n_1|, |2n_2| < q_{k+1}$, one has
\begin{equation}\label{e.gapbound}
\frac{1}{2q_{k+1}}<\frac{1}{q_k+q_{k+1}}< |z_{n_1}-z_{n_2}|.
\end{equation}

Fix $\gamma \in (0, 1/2)$ and take any $\tilde \gamma \in (\gamma, 1/2)$. For any $z \in \mathbb{T}$ and $\varepsilon > 0$, consider the function $g : \mathbb{T} \to \mathbb{R}$ defined by
\begin{equation}\label{e.formofg}
g(x)=\left\{
       \begin{array}{ll}
         \varepsilon^{\tilde\gamma-1}(z+\varepsilon-x), & \hbox{for \ $x\in [z, z+\varepsilon]$;} \\
         \varepsilon^{\tilde\gamma-1}(x+\varepsilon-z), & \hbox{for \ $x\in [z-\varepsilon, z]$;} \\
         0, & \hbox{for \  $x\not\in [z-\varepsilon, z+\varepsilon]$.}
       \end{array}
     \right.
\end{equation}
Notice that $g(z)=\varepsilon^{\tilde \gamma}$, and that $g$ is H\"older continuous with H\"older exponent $\tilde \gamma$ and constant one.

Let $g_n$ be the function of the form (\ref{e.formofg}) for the parameters $z=z_n$ and $\varepsilon=\frac{1}{100q_{k+1}}$, where $|2n| \in [q_k, q_{k+1})$. Then due to \eqref{e.gapbound}, if $|2n|, |2m| \in [q_k, q_{k+1})$, the supports of the functions $g_n$ and $g_m$ do not intersect. Therefore, for any $x \in \mathbb{T}$,
$$
\sum_{n \in \mathbb{Z}} g_n(x) < \sum_{k=1}^\infty \frac{2}{(100q_{k+1})^{\tilde \gamma}} < \infty.
$$
Hence for any sequence $\{ \mu_n \}_{n \in \mathbb{Z}}$, $\mu_n \in [0,1]$, the series $\sum_{n \in \mathbb{Z}} \mu_n g_n$ converges uniformly. We claim that it converges to a H\"older function\footnote{This can be seen from general principles, due to a relation between the modulus of continuity of a function and its degree of approximation by Lipschitz functions; see, e.g., \cite[Proposition 1.10]{BL}, but in our particular case a direct calculation seems to be a shorter path.} with H\"older exponent $\gamma$. Indeed, denote $G = \sum_{n \in \mathbb{Z}} \mu_n g_n$, and take any two points $x, y \in \mathbb{T}$. For some $k \in \mathbb{N}$, we have
$$
\frac{1}{q_{k+1}} \le |x-y| < \frac{1}{q_k}.
$$
Notice that
\begin{align*}
|G(x) - G(y)| & \le \sum_{n \in \mathbb{Z}} \mu_n |g_n(x) - g_n(y)| \\
& \le \sum_{n \in [-q_k/2, q_k/2]} |g_n(x) - g_n(y)| + \sum_{|n| > q_k/2} |g_n(x) - g_n(y)| \\
& \le |x-y| \left( \sum_{s=1}^k \frac{(100 q_s)^{-\tilde \gamma}}{(100q_s)^{-1}} \right) + \sum_{s = k+1}^\infty(100q_s)^{-\tilde \gamma} \\
& \le |x-y| C_1 \cdot q_k^{1 - \tilde \gamma} + C_2 q_{k+1}^{-\tilde \gamma}.
\end{align*}
We need to show that if $\gamma > \tilde \gamma$, then
\begin{equation}\label{e.holder}
|G(x) - G(y)| \le |x-y|^\gamma
\end{equation}
for sufficiently close $x$ and $y$. Notice that if $|x - y| = \frac{1}{q_{k+1}}$, then the inequality
\begin{equation}\label{e.testineq}
|x-y| C_1 \cdot q_k^{1 - \tilde \gamma} + C_2 q_{k+1}^{-\tilde \gamma} \le |x-y|^\gamma
\end{equation}
is equivalent to
$$
\frac{1}{q_{k+1}} C_1 \cdot q_k^{1 - \tilde \gamma} + C_2 q_{k+1}^{-\tilde \gamma} \le \frac{1}{q_{k+1}^\gamma},
$$
or
$$
C_1 q_{k+1}^{\gamma - 1} q_k^{1 - \tilde \gamma} + C_2 q_{k+1}^{\gamma - \tilde \gamma} \le 1.
$$
Since $0 < \gamma < \tilde \gamma$, this holds for all large enough $k \in \mathbb{N}$.

On the other hand, if $|x - y| = \frac{1}{q_k}$, then \eqref{e.testineq} is equivalent to
$$
\frac{1}{q_{k}} C_1 \cdot q_k^{1 - \tilde \gamma} + C_2 q_{k+1}^{-\tilde \gamma} \le \frac{1}{q_{k}^\gamma},
$$
or
$$
C_1 q_k^{\gamma - \tilde \gamma} + C_2 q_{k+1}^{-\tilde \gamma} q_k^{\gamma} \le 1.
$$
Once again, since $0 < \gamma < \tilde \gamma$, this holds for all large enough $k \in \mathbb{N}$.

Finally, since the function $t \mapsto t^\gamma$ is concave on the interval $t \in \left[ \frac{1}{q_{k+1}}, \frac{1}{q_k} \right]$, the inequality \eqref{e.testineq} (and hence \eqref{e.holder}) holds for all $x, y \in \mathbb{T}$ with $|x - y| \in \left[ \frac{1}{q_{k+1}}, \frac{1}{q_k} \right]$ for all sufficiently large $k \in \mathbb{N}$, and hence for all sufficiently close $x, y$.

Let $\{ \xi_n \}_{n \in \mathbb{Z}}$ be a family of independent random variables with a distribution whose density is given by some function $r(x)$, and let $f : \mathbb{T} \to \mathbb{R}$ be a H\"older continuous function, $f \in C^{\gamma}(\mathbb{T}, \mathbb{R})$. Then the random function $g : \mathbb{T} \to \mathbb{R}$, $g = f + \sum_{n \in \mathbb{Z}} \xi_n g_n$ also belongs to $C^{\gamma}(\mathbb{T}, \mathbb{R})$. In order to apply Theorem~\ref{t.bs}, we need to check that our condition \eqref{e.condkappa} implies the condition from Theorem~\ref{t.bs}. Set $\chi_n = f(z_n)$. We need to ensure that
\begin{equation}\label{e.111}
\sum_{n \in \mathbb{N}} a_n^{-1/2} \exp \left( -\delta \sum_{s=1}^n |a_s|^2 \right) < \infty
\end{equation}
and
\begin{equation}\label{e.222}
\sum_{n \le 0} a_n^{-1/2} \exp \left(-\delta \sum_{s=n}^0 |a_s|^2 \right) < \infty,
\end{equation}
where $a_n = \frac{1}{(100q_{k+1})^{\tilde \gamma}}$, if $|2n| \in [q_k, q_{k+1})$. Notice that $\delta$ here depends on the density $r$, and below we will choose $r$ to make $\delta$ sufficiently large. We have
$$
\sum_{s = q_k/2}^{q_{k+1}/2} a_s^2 \approx \frac{1}{10000q_{k+1}^{2\tilde \gamma}} \left( \frac{1}{2}q_{k+1} - \frac{1}{2}q_{k} \right) \approx q_{k+1}^{1 - 2 \tilde \gamma}.
$$
Therefore,
\begin{align*}
\sum_{n \in \mathbb{N}} & a_n^{-1/2} \exp \left( -\delta \sum_{s=1}^n |a_s|^2 \right) = \sum_{k=1}^\infty \sum_{n=q_k/2}^{q_{k+1}/2} a_n^{-1/2} \exp \left( -\delta \left[ \sum_{s=1}^{q_k/2-1}a_s^2 \right] - \delta \left[ \sum_{s=q_k/2}^n a_s^2 \right] \right) \\
& = \sum_{k=1}^\infty \exp \left( -\delta \sum_{s=1}^{q_k/2-1} a_s^2 \right) \cdot (100q_{k+1})^{\tilde \gamma/2} \sum_{n=q_k/2}^{q_{k+1}/2} \exp \left( -\delta (n-\frac{1}{2}q_k)(100q_{k+1})^{2\tilde \gamma} \right) \\
& \le C \sum_{k=1}^\infty (100q_{k+1})^{\tilde \gamma/2} \exp \left( -\delta \sum_{s=1}^{q_k/2-1} a_s^2 \right) \\
& \le C\sum_{k=1}^\infty(100q_{k+1})^{\tilde \gamma/2}\exp\left(-\tilde \delta q_k^{1-2\tilde\gamma}\right) < C'\sum_{k=1}^\infty q_{k+1}^{\tilde\gamma/2}\exp\left(-\kappa q_k^{1-2\tilde\gamma}\right) < \infty,
\end{align*}
due to \eqref{e.condkappa} and the fact that the coefficient $\delta > 0$ in \eqref{e.111} and \eqref{e.222} depends on the Fourier transform $\hat r$ (see Remark~\ref{r.hatr}), and by appropriate choice of the density $r$, it can be made as large as needed. This shows \eqref{e.111}. The proof of \eqref{e.222} is completely similar.
\end{proof}

\section*{Acknowledgment}

We are grateful to Valmir Bucaj, Victor Chulaevsky and Jake Fillman for useful discussions.

\end{document}